\documentclass[a4paper,12pt]{amsart}
\usepackage{amsfonts}
\usepackage{amsmath}
\usepackage{amssymb}
\usepackage{mathrsfs}
\usepackage[colorlinks]{hyperref}
\renewcommand\eqref[1]{(\ref{#1})} %Need with hyperref
\numberwithin{equation}{section}
\theoremstyle{plain}
\newtheorem{thm}{Theorem}[section]

\newtheorem{lem}[thm]{Lemma}

\theoremstyle{definition}
\newtheorem{defn}[thm]{Definition}
\newtheorem{rem}[thm]{Remark}
\newtheorem{ex}[thm]{Example}

\begin{document}
\title[Isoperimetric inequalities for the logarithmic potential]
{Isoperimetric inequalities for the logarithmic potential operator}

\author[Michael Ruzhansky]{Michael Ruzhansky}
\address{
  Michael Ruzhansky:
  \endgraf
  Department of Mathematics
  \endgraf
  Imperial College London
  \endgraf
  180 Queen's Gate, London SW7 2AZ
  \endgraf
  United Kingdom
  \endgraf
  {\it E-mail address} {\rm m.ruzhansky@imperial.ac.uk}
  }
\author[Durvudkhan Suragan]{Durvudkhan Suragan}
\address{
  Durvudkhan Suragan:
  \endgraf
  Institute of Mathematics and Mathematical Modelling
  \endgraf
  125 Pushkin str., 050010 Almaty, Kazakhstan
  \endgraf
  and
  \endgraf
  Department of Mathematics
  \endgraf
  Imperial College London
  \endgraf
  180 Queen's Gate, London SW7 2AZ
  \endgraf
  United Kingdom
  \endgraf
  {\it E-mail address} {\rm d.suragan@imperial.ac.uk}
  }

\thanks{The authors were supported in parts by the EPSRC
 grant EP/K039407/1 and by the Leverhulme Grant RPG-2014-02,
 as well as by the MESRK grant 5127/GF4.}


\begin{abstract}
In this paper we prove that the disc is a maximiser of the Schatten $p$-norm of the logarithmic potential operator among all domains of a given measure in $\mathbb R^{2}$, for all even integers $2\leq p<\infty$. We also show that the equilateral triangle has the largest Schatten $p$-norm among all triangles of a given area. For the logarithmic potential operator on bounded open or triangular domains, we also obtain analogies of the Rayleigh-Faber-Krahn or P{\'o}lya inequalities, respectively. The logarithmic potential operator can be related to a nonlocal boundary value problem for the Laplacian, so we obtain isoperimetric inequalities for its eigenvalues as well.
\end{abstract}
     \maketitle
\section{Introduction}
\label{intro}

Let $\Omega\subset \mathbb R^{2}$ be an open bounded set. We consider the logarithmic potential operator on $L^{2}(\Omega)$ defined by
\begin{equation}
\mathcal{L}_{\Omega}f(x):= \int_{\Omega}\frac{1}{2\pi}\ln\frac{1}{|x-y|}f(y)dy,\quad f\in L^{2}(\Omega),\label{3}
\end{equation}
where $\ln$ is the natural logarithm and $|x-y|$ is the standard Euclidean
distance between $x$ and $y$.
Clearly, $\mathcal{L}_{\Omega}$ is a compact and self-adjoint operator.
Therefore, all of its eigenvalues and characteristic numbers are discrete and real.
We recall that the characteristic numbers are the inverses of the eigenvalues. The characteristic numbers of $\mathcal{L}_{\Omega}$ may be enumerated in ascending order of their modulus,
$$|\mu_{1}(\Omega)|\leq|\mu_{2}(\Omega)|\leq...$$
where $\mu_{i}(\Omega)$ is repeated in this series according to its multiplicity. We denote the
corresponding eigenfunctions by $u_{1}, u_{2},...,$ so that for each characteristic number $\mu_{i}$ there is a unique corresponding (normalised) eigenfunction $u_{i}$,
$$u_{i}=\mu_{i}(\Omega)\mathcal{L}_{\Omega}u_{i},\,\,\,\, i=1,2,....$$
\medskip
It is known, see for example Mark Kac \cite{Kac1} (see also \cite{KS1}),
that the equation
$$
u(x)=\mathcal{L}_{\Omega} f(x)=\int_{\Omega}\frac{1}{2\pi}\ln\frac{1}{|x-y|}f(y)dy
$$
is equivalent to the equation
\begin{equation}\label{15}
-\Delta u(x)=f(x), \,\,\,\ x\in\Omega,
\end{equation}
with the nonlocal integral boundary condition
\begin{equation}
-\frac{1}{2}u(x)+\int_{\partial\Omega}\frac{\partial}{\partial n_{y}}\frac{1}{2\pi}\ln\frac{1}{|x-y|}u(y)d S_{y}-
\int_{\partial\Omega}\frac{1}{2\pi}\ln\frac{1}{|x-y|}\frac{\partial u(y)}{\partial n_{y}}d S_{y}=0,\,\,
x\in\partial\Omega,
\label{16}
\end{equation}
where $\frac{\partial}{\partial n_{y}}$ denotes the outer normal
derivative at a point $y$ on the boundary $\partial\Omega$, which is assumed piecewise $C^1$ here.

In general, the boundary value problem \eqref{15}-\eqref{16} has several interesting applications (see, Kac \cite{Kac1,Kac2}, Saito \cite{Sa} and \cite{KS1}).

Spectral properties of the logarithmic potential have been considered in many papers (see \cite{AKL}, \cite{AK}, \cite{BS}, \cite{D12}, \cite{Kac3}, \cite{Tr67}, \cite{Tr69}). In this paper we are interested in isoperimetric inequalities of the logarithmic potential $\mathcal{L}_{\Omega}$, that is also, in isoperimetric inequalities of the nonlocal Laplacian \eqref{15}-\eqref{16}. For a recent general review of
isoperimetric inequalities for the Dirichlet, Neumann and other Laplacians we refer to Benguria, Linde and Loewe in \cite{Ben}.
Isoperimetric inequalities for Schatten norms for double layer potentials have been recently considered
by Miyanishi and Suzuki \cite{MS}.

In Rayleigh's famous book ``Theory of Sound'' (first published in 1877), by using some explicit computation and physical interpretations, he stated that the disc minimises (among all
domains of the same area) the first eigenvalue of the Dirichlet Laplacian. The proof of this conjecture
was obtained about 50 years later, simultaneously (and independently) by G. Faber and E. Krahn. Nowadays, the
Rayleigh-Faber-Krahn inequality has been
established for many other operators; see e.g. \cite{He} for
further references (see also \cite{Ban80} and \cite{Po1}). Among other things, in this paper we also
prove the Rayleigh-Faber-Krahn inequality for the integral operator $\mathcal{L}_{\Omega}$, i.e.
it is proved that the disc is a minimiser of the first eigenvalue of the Laplacian \eqref{15}-\eqref{16} among
all domains of a given measure in $\mathbb R^{2}$.

By using the Feynman-Kac formula and spherical rearrangement Luttinger \cite{Lu} proved that the disc $D$ is a maximiser
of the partition function of the Dirichlet Laplacian among all domains of the same area as $D$ for all positive values of time, i.e.
$$
\sum_{i=1}^{\infty}\exp(-t\mu_{i}^{\mathcal{D}}(\Omega))\leq \sum_{i=1}^{\infty}\exp(-t\mu_{i}^{\mathcal{D}}(D)),\quad \forall t>0 ,\,\,|\Omega|=|D|,
$$
where $\mu_{i}^{\mathcal{D}}, i=1,2,\ldots,$ are the characteristic numbers of the Dirichlet Laplacian.
From here by using the Mellin transform one obtains
\begin{equation}\sum_{i=1}^{\infty}\frac{1}{[\mu_{i}^{\mathcal{D}}(\Omega)]^{p}}\leq \sum_{i=1}^{\infty}\frac{1}{[\mu_{i}^{\mathcal{D}}(D)]^{p}},\quad
|\Omega|=|D|, \label{eq5}\end{equation}
when $p>1$, $\Omega\subset \mathbb R^{2}$.
We prove an analogy of this Luttinger's inequality for the integral operator $\mathcal{L}_{\Omega}$. In our
note \cite{Ruzhansky-Suragan:UMN} we obtained similar results for convolution type integral operators with positive nonincreasing kernels. In the present setting the main difficulty arises from the fact that the logarithmic potential is not positive and that we can not use the Brascamp-Lieb-Luttinger type rearrangement inequalities directly.

In Section \ref{SEC:result} we present main results of this paper.
Their proofs will be given in Section \ref{SEC:proof} and Section \ref{sec:2}.
In Section 5 we discuss shortly about isoperimetric inequalities for polygons and show that the Schatten $p$-norm is maximised on the equilateral triangle centred at the origin among all triangles of a given area.

\medskip
The authors would like to thank Grigori Rozenblum for comments.

\section{Main results and examples}
\label{SEC:result}

Let $H$ be a separable Hilbert space. By $\mathcal{S}^{\infty}(H)$
we denote the space of compact operators $P:H\rightarrow H$.
Recall that the singular values $\{s_{n}\}$ of $P\in \mathcal{S}^{\infty}(H)$
are the eigenvalues of the positive operator $(P^{*}P)^{1/2}$ (see e.g. \cite{GK}). The Schatten $p$-classes are defined as
$$
\mathcal{S}^{p}(H):=\{P \in \mathcal{S}^{\infty}(H): \{s_{n}\}\in \ell^{p}\},\quad 1\leq p<\infty.
$$
In $\mathcal{S}^{p}(H)$ the Schatten $p$-norm of the operator $P$ is defined by
\begin{equation}
\| P\|_{p}:=\left(\sum_{n=1}^{\infty}s_{n}^{p}\right)^{\frac{1}{p}}, \quad 1\leq p <\infty.\label{1}
\end{equation}
For $p=\infty$, we can set
$$
\| P\|_{\infty}:=\| P\|
$$
to be the operator norm of $P$ on $H$. As outlined in the introduction, we assume that $\Omega\subset \mathbb R^{2}$
is an open bounded set and we consider the logarithmic potential operator
on $L^{2}(\Omega)$ of the form
\begin{equation}
\mathcal{L}_{\Omega}f(x)=\int_{\Omega}\frac{1}{2\pi}\ln\frac{1}{|x-y|}f(y)dy,\quad f\in L^{2}(\Omega).
\label{n16}
\end{equation}
We also assume that the operator $\mathcal{L}_{\Omega}$ is positive:

\begin{rem}\label{REM:positivity}
In Landkof \cite[Theorem 1.16, p.80]{Lan66} the positivity of the operator $\mathcal{L}_\Omega$ is proved in domains $\overline{\Omega}\subset U,$ where $U$ is the unit disc. In general, $\mathcal{L}_\Omega$ is not a positive operator. For any bounded open domain $\Omega$ the logarithmic potential operator $\mathcal{L}_\Omega$ can have at most one negative eigenvalue, see
Troutman \cite{Tr67} (see also Kac \cite{Kac3}).
\end{rem}

Note that for positive self-adjoint operators the singular values equal the eigenvalues.
It is known that $\mathcal{L}_{\Omega}$ is a Hilbert-Schmidt operator.
By $|\Omega|$ we will denote the Lebesque measure of $\Omega$.

\begin{thm} \label{THM:main}
Let $D$ be a disc centred at the origin.
Then
\begin{equation}
\|\mathcal{L}_{\Omega}\|_{p}\leq  \|\mathcal{L}_{D}\|_{p}
\label{3}
\end{equation}
for any even integer $2\leq p<\infty$ and any bounded open domain $\Omega$ with $|\Omega|=|D|.$
\end{thm}

Note that for even integers $p$ we do not need to assume
the positivity of the logarithmic potential operator.
For odd integers we have the following:

\begin{thm}\label{THM:main2}
Let $D$ be a disc centred at the origin and let $\Omega$
be a bounded open domain with $|\Omega|=|D|.$
Assume that the logarithmic potential operator is positive for
$\Omega$ and $D$.
Then
\begin{equation}
\|\mathcal{L}_{\Omega}\|_{p}\leq  \|\mathcal{L}_{D}\|_{p}
\label{3-1}
\end{equation}
for any integer $2\leq p<\infty$.
\end{thm}

Let us give several examples calculating explicitly values of the right hand side of \eqref{3} for different values of $p$.
\begin{ex}
Let $D\equiv U$ be the unit disc.
Then by Theorem \ref{THM:main} we have
\begin{equation}
\|\mathcal{L}_{\Omega}\|_{p}\leq
\|\mathcal{L}_{U}\|_{p}=
\left(\sum_{m=1}^{\infty}\frac{3}{j_{0,m}^{2p}}
+\sum_{l=1}^{\infty}\sum_{m=1}^{\infty}\frac{2}{j_{l,m}^{2p}}
\right)^{\frac{1}{p}},
\label{ex2}
\end{equation}
for any even $2\leq p < \infty$ and any bounded open domain $\Omega$ with $|\Omega|=|D|.$
Here $j_{km}$ denotes the $m^{th}$ positive zero of the Bessel function  $J_{k}$ of the first kind of order $k$.
\end{ex}
The right hand sight of the formula \eqref{ex2} can be confirmed by a direct calculation of the logarithmic potential eigenvalues in the unit disc, see Theorem 3.1 in \cite{AKL}.

\medskip
We also obtain the following Rayleigh-Faber-Krahn inequality when $p=\infty$:

\begin{thm} \label{THM:1} The disc $D$ is a minimiser of the characteristic number of the logarithmic potential $\mathcal{L}_\Omega$ with the smallest modulus among all domains of a given measure, that is,
$$\| \mathcal{L}_\Omega\|\leq \| \mathcal{L}_D\|$$
for an arbitrary bounded open domain $\Omega\subset \mathbb R^{2}$ with $|\Omega|=|D|.$
\end{thm}

\begin{ex}
Let $D\equiv U$ be the unit disc.
Then by Theorem \ref{THM:1} we have
\begin{equation}
\| \mathcal{L}_\Omega\|
\leq \| \mathcal{L}_U\| =\frac{1}{j_{01}^{2}}
\label{ex3}
\end{equation}
for any bounded open domain $\Omega$ with $|\Omega|=|D|.$ Here $\|\cdot\|$ is the operator norm on the space $L^2$.
\end{ex}
From Corollary 3.2 in \cite{AKL} we calculate explicitly the operator norm in the right hand sight of
\eqref{ex3}.

\section{Proof of Theorem \ref{THM:1}}
\label{sec:2}

Let us first prove Theorem \ref{THM:1}. To do it we first prove the following:

\begin{lem}\label{lem:1}
The characteristic number $\mu_{1}$ of the logarithmic potential $\mathcal{L}_\Omega$ with the smallest modulus is simple,
and the corresponding eigenfunction $u_{1}$ can be chosen nonnegative.
\end{lem}

\begin{proof}%[Proof of Lemma \ref{lem:1}]
The eigenfunctions of the logarithmic potential $\mathcal{L}_\Omega$
may be chosen to be real as its kernel is real. First let us prove that $u_{1}$ cannot change sign in the domain
$\Omega$, that is,
$$u_{1}(x)u_{1}(y)=|u_{1}(x)u_{1}(y)|,\,\,x,\,y\in\Omega.$$

Indeed, in the opposite case, by virtue of the continuity of the function
$u_{1}(x)$, there would be neighborhoods $U(x_{0},r)\subset \Omega$ such that

$$|u_{1}(x)u_{1}(y)|> u_{1}(x)u_{1}(y),\quad x,y\in U(x_{0},r)\subset\Omega.$$
On the other hand we have
\begin{equation}
\int_{\Omega}\frac{1}{2\pi}\ln\frac{1}{|x_{0}-z|}
\frac{1}{2\pi}\ln\frac{1}{|z-x_{0}|}dz>0, \quad x_{0}\in\Omega.
\end{equation}
From here by continuity it is simple to check that there exists $\rho>0$ such that
\begin{equation}\label{n0}
\int_{\Omega}\frac{1}{2\pi}\ln\frac{1}{|x-z|}
\frac{1}{2\pi}\ln\frac{1}{|z-y|}dz>0, \quad x,y\in U(x_{0},\rho)\subset U(x_{0},r).
\end{equation}
Now let us introduce a new function
\begin{equation}
\widetilde{u}_{1}(x):=\left\{
\begin{array}{ll}
    |u_{1}(x)|,\,\,x\in U(x_{0},\rho),\\
    u_{1}(x), \,\,x\in \Omega\backslash U(x_{0},\rho). \\
\end{array}
\right.
\label{EQ:fs}
\end{equation}
Then we obtain
\begin{multline} \label{n2}
\frac{(\mathcal{L}^{2}_\Omega\widetilde{u}_{1},\,\widetilde{u}_{1})}{\| \widetilde{u}_{1}\|^{2}}=
\frac{1}{\| \widetilde{u}_{1}\|^{2}}\int_{\Omega}
\int_{\Omega}\int_{\Omega}\frac{1}{2\pi}\ln\frac{1}
{|x-z|}\frac{1}{2\pi}\ln\frac{1}{|z-y|}dz\widetilde{u}_{1}(x)
\widetilde{u}_{1}(y) dxdy \\
>\frac{1}{\| u_{1}\|^{2}}\int_{\Omega}\int_{\Omega}\int_{\Omega}
\frac{1}{2\pi}\ln\frac{1}{|x-z|}\frac{1}{2\pi}\ln\frac{1}{|z-y|}dz u_{1}(x)u_{1}(y)dxdy=\frac{1}{\mu_{1}^{2}},
\end{multline}
where
$\mu_{1}^{2}$ is the smallest characteristic number of $\mathcal{L}^{2}_\Omega$ and $u_{1}$
is the eigenfunction corresponding to $\mu_{1}^{2}$, i.e.
$$u_{1}=\mu_{1}^{2}\mathcal{L}^{2}_\Omega u_{1}.$$
Therefore, by the variational principle we also have
\begin{equation}\label{n3}
\frac{1}{\mu_{1}^{2}}=\sup_{f\in L^{2}(\Omega), f\not=0}\frac{(\mathcal{L}^{2}_\Omega f,f)}{\| f\|^{2}}.
\end{equation}
This means that the strong inequality \eqref{n2} contradicts the variational
principle \eqref{n3} because
$\|\widetilde{u}_{1}\|_{L^2}=\|{u}_{1}\|_{L^2}<\infty$.

Since $u_{1}$ is nonnegative it follows that
$\mu_{1}$ is simple. Indeed, if there were an eigenfunction
$v_{1}$ linearly independent of $u_{1}$ and corresponding
to $\mu_{1}$, then for all real $c$ the linear combination
$u_{1}+c v_{1}$ also would be an eigenfunction corresponding
to $\mu_{1}$ and therefore, by what has been proved, it could not become negative
in $\Omega$. As $c$ is arbitrary, this is impossible.
\end{proof}

\begin{proof}[Proof of Theorem \ref{THM:1}]
Let $\Omega$ be a bounded open set in $\mathbb R^{2}.$
Its symmetric rearrangement $\Omega^{\ast}\equiv D$ is an open disc centred at $0$
with the measure equal to the measure of $\Omega,$ i.e. $|D|=|\Omega|$.
Let $u$ be a nonnegative measurable function in $\Omega$,
such that all its positive level sets have finite measure.
In the definition of the symmetric-decreasing rearrangement of $u$
one can use the so-called layer-cake decomposition (see \cite{LL}),
which expresses a nonnegative function $u$ in terms of its level sets as
\begin{equation}
u(x)=\int_0^{\infty}\chi_{\left\{ u(x)>t\right\}}dt,
\end{equation}
where $\chi$ is the characteristic function of the corresponding domain.

\begin{defn} Let $u$ be a nonnegative measurable function in $\Omega$. The function
\begin{equation}
u^{\ast}(x):=\int_0^{\infty}\chi_{\left\{u(x)>t\right\}^{\ast}\label{2}
}dt
\end{equation}
is called the (radially) symmetric-decreasing rearrangement of a nonnegative measurable function $u$.
\end{defn}

 As in the proof of Lemma \ref{lem:1}
$\mu_{1}^{2}(\Omega)$ is the smallest characteristic number of $\mathcal{L}^{2}_\Omega$ and $u_{1}$
is the eigenfunction corresponding to $\mu_{1}^{2}$, i.e.
$$u_{1}=\mu_{1}^{2}(\Omega)\mathcal{L}^{2}_\Omega u_{1}.$$
By Lemma \ref{lem:1} the first characteristic number $\mu_{1}$ of the operator
 $\mathcal{L}_\Omega$ is simple; the corresponding
eigenfunction $u_{1}$ can be chosen positive in $\Omega$, and in view of Lemma
\ref{lem:1} we can apply the above construction to the first eigenfunction $u_{1}$.
Recall\footnote{For the proof of the rearrangement inequality \eqref{n4}
for the logarithmic kernel see Lemma \ref{lem:Riesz}. The proof is the same with
the difference that in this case the symmetric decreasing rearrangement is used instead of the Steiner symmetrization.}
 the rearrangement inequality for the logarithmic kernel (cf. Lemma 2 in \cite{CL}):

\begin{multline}
\int_{\Omega}\int_{\Omega}\int_{\Omega}
u_{1}(y)\frac{1}{2\pi}\ln\frac{1}{|y-z|}
\frac{1}{2\pi}\ln\frac{1}{|z-x|}u_{1}(x)dzdydx\leq\\
\int_{D}
\int_{D}\int_{D}
u_{1}^{\ast}(y)\frac{1}{2\pi}\ln\frac{1}{|y-z|}
\frac{1}{2\pi}\ln\frac{1}{|z-x|}u_{1}^{\ast}(x)dzdydx.\label{n4}
\end{multline}
In addition, for each nonnegative function $u\in L^{2}(\Omega)$ we have
\begin{equation}
\| u\|_{L^{2}(\Omega)}=\| u^{\ast}\|_{L^{2}(D)}.\label{n5}
\end{equation}

Therefore, from \eqref{n4}, \eqref{n5} and the variational principle for the positive operator $\mathcal{L}^{2}_D$, we get
$$
\mu_{1}^{2}(\Omega) =\frac{\int_{\Omega}|u_{1}(x)|^{2}dx}{\int_{\Omega}\int_{\Omega}\int_{\Omega}
u_{1}(y)\frac{1}{2\pi}\ln\frac{1}{|y-z|}
\frac{1}{2\pi}\ln\frac{1}{|z-x|}u_{1}(x)dzdydx}\geq$$
$$\frac{\int_{D}|u^{\ast}_{1}(x)|^{2}dx}
{\int_{D}\int_{D}\int_{D}
u^{\ast}_{1}(y)\frac{1}{2\pi}\ln\frac{1}{|y-z|}
\frac{1}{2\pi}\ln\frac{1}{|z-x|}u^{\ast}_{1}(x)dzdydx}\geq
$$
$$
\inf_{v\in
L^{2}(D), v\not=0}\frac{\int_{D}|v(x)|^{2}dx}{\int_{D}\int_{D}\int_{D}
v(y)\frac{1}{2\pi}\ln\frac{1}{|y-z|}
\frac{1}{2\pi}\ln\frac{1}{|z-x|}v(x)dzdydx}=\mu_{1}^{2}(D).
$$
Finally, note that $0$ is not a characteristic number of $\mathcal{L}_D$ (cf. Corollary 1 in \cite{Tr67}). Therefore,
$$0<|\mu_{1}(D)|.$$
This completes the proof.
\end{proof}

\section{Proofs of Theorem \ref{THM:main} and Theorem \ref{THM:main2}}
\label{SEC:proof}
First we prove the Brascamp-Lieb-Luttinger type rearrangement inequality for the logarithmic kernel (cf. \cite{BLL}).

\begin{lem}\label{lem:n1} Let $D$ be a disc centred at the origin. Then
\begin{multline}
\int_{\Omega}...
\int_{\Omega}\frac{1}{2\pi}\ln\frac{1}{|y_{1}-y_{2}|}...
\frac{1}{2\pi}\ln\frac{1}{|y_{p}-y_{1}|}dy_{1}...dy_{p}\leq
\\
\int_{D}...
\int_{D}\frac{1}{2\pi}\ln\frac{1}{|y_{1}-y_{2}|}...
\frac{1}{2\pi}\ln\frac{1}{|y_{p}-y_{1}|}dy_{1}...dy_{p},\label{eqn16}
\end{multline}
for any $p=2,3,\ldots,$ and for any bounded open set $\Omega$ with $|\Omega|=|D|.$
\end{lem}
\begin{proof} [Proof of Lemma \ref{lem:n1}]
Here we prove it for $p=2$ and the proof is based on the proof of Lemma 2 in \cite{CL}.
The proof for arbitrary $p$ is essentially the same as the case $p=2$.
Let us fix $r_{0}>0$ and consider the function
\begin{equation}
f(r):=\left\{
\begin{array}{ll}
    \frac{1}{2\pi}\ln\,\frac{1}{r},\,\,r\leq r_{0},\\
    \frac{1}{2\pi}\ln\,\frac{1}{r_0}-\frac{1}{2\pi}\int_{r_{0}}^{r}s^{-1}\frac{1+r_{0}^{2}}{1+s^{2}}ds, \,\,r>r_{0}. \\
\end{array}
\right.
\label{EQ:fs1}
\end{equation}
Let us show that the function $f(r)$ is strictly decreasing and has a limit as $r\rightarrow \infty$.
If $r\leq r_{0}$ then
$$
f(r_{1})=\frac{1}{2\pi}\ln\,\frac{1}{r_{1}}> \frac{1}{2\pi}\ln\,\frac{1}{r_{2}}=f(r_{2})
$$
for $r_{1}< r_{2}$.
If $r> r_{0}$ then
\begin{multline}\label{002}
f(r)=\frac{1}{2\pi}\ln\,\frac{1}{r_0}-\frac{1}{2\pi}\int_{r_{0}}^{r}s^{-1}
\frac{1+r_{0}^{2}}{1+s^{2}}ds=
\\
\frac{1}{2\pi}\ln\,\frac{1}{r_0}-
\frac{1}{2\pi}(1+r_{0}^{2})[\ln r-\frac{1}{2}\ln(1+r^{2})-\ln r_{0}+\frac{1}{2}\ln(1+r_{0}^{2})].
\end{multline}
Thus
$f(r_{1})>f(r_{2})$
for $r_{1}< r_{2}$, that is,
$f(r)$ is  strictly decreasing.
From \eqref{002} it is easy to see
that
\begin{equation}
\lim_{r\rightarrow\infty} f(r)=\frac{1}{2\pi}\ln\,\frac{1}{r_0}-
\frac{1}{2\pi}(1+r_{0}^{2})[-\ln r_{0}+\frac{1}{2}\ln(1+r_{0}^{2})].
\end{equation}
We use the notation
$$
f_{\infty}:=\frac{1}{2\pi}\ln\,\frac{1}{r_0}-
\frac{1}{2\pi}(1+r_{0}^{2})[-\ln r_{0}+\frac{1}{2}\ln(1+r_{0}^{2})].
$$
By construction $\frac{1}{2\pi}\ln\,\frac{1}{r}-f(r)$ is decreasing.
Thus if we define
$$h_{1}(r)=f(r)-f_{\infty}$$
we have the decomposition
$$\frac{1}{2\pi}\ln\,\frac{1}{r}=h_{1}(r)+h_{2}(r)$$
where $h_{1}$ is positive strictly decreasing function and $h_{2}$
is decreasing. Hence by the Brascamp-Lieb-Luttinger rearrangement inequality we have
\begin{multline}
\int_{\Omega}\int_{\Omega}
h_{1}(|y_{1}-y_{2}|)h_{1}(|y_{2}-y_{1}|)dy_{1}dy_{2}\leq\\
\int_{D}
\int_{D}h_{1}(|y_{1}-y_{2}|)h_{1}(|y_{2}-y_{1}|)dy_{1}dy_{2}\label{01}
\end{multline}
and
\begin{multline}
\int_{\Omega}\int_{\Omega}
h_{2}(|y_{1}-y_{2}|)h_{2}(|y_{2}-y_{1}|)dy_{1}dy_{2}\leq \\ \int_{D}
\int_{D}h_{2}(|y_{1}-y_{2}|)h_{2}(|y_{2}-y_{1}|)dy_{1}dy_{2}.\label{01}
\end{multline}
Thus it remains to show
\begin{multline}
\int_{\Omega}\int_{\Omega}
h_{1}(|y_{1}-y_{2}|)h_{2}(|y_{2}-y_{1}|)dy_{1}dy_{2}\leq \\\int_{D}
\int_{D}h_{1}(|y_{1}-y_{2}|)h_{2}(|y_{2}-y_{1}|)dy_{1}dy_{2}\label{01}
\end{multline}
which does not follow directly from the Brascamp-Lieb-Luttinger rearrangement inequality since $h_{2}$ is not positive.
Define for $R>0$
\begin{equation}
q_{R}(r):=\left\{
\begin{array}{ll}
    h_{2}(r)-h_{2}(R),\,\,r\leq R,\\
    0, \,\,r>R, \\
\end{array}
\right.
\label{EQ:fs2}
\end{equation}
and note that by monotone convergence
\begin{equation}
I_{\Omega}(h_{1},h_{2})=\lim_{R\rightarrow\infty}[I_{\Omega}(h_{1},q_{R})+h_{2}(R)\int_{\Omega}\int_{\Omega}h_{1}(|y_{1}-y_{2}|)dy_{1}dy_{2}]
\end{equation}
with the notation
\begin{equation}
I_{\Omega}(f,g)=\int_{\Omega}\int_{\Omega}
f(|y_{1}-y_{2}|)g(|y_{2}-y_{1}|)dy_{1}dy_{2}.
\end{equation}
Since $h_{1}$ and $q_{R}$ are positive and nonincreasing
$$I_{\Omega}(h_{1}, q_{R})\leq I_{D}(h_{1},q_{R})$$
by the Brascamp-Lieb-Luttinger rearrangement inequality.
Noting that
$$\int_{\Omega}\int_{\Omega}h_{1}(|y_{1}-y_{2}|)dy_{1}dy_{2}\leq\int_{D}\int_{D}h_{1}(|y_{1}-y_{2}|)dy_{1}dy_{2}$$
we obtain
\begin{multline}
I_{\Omega}(h_{1},h_{2})=
\lim_{R\rightarrow\infty}[I_{\Omega}(h_{1},q_{R})+h_{2}(R)\int_{\Omega}\int_{\Omega}h_{1}(|y_{1}-y_{2}|)dy_{1}dy_{2}]\leq\\
\lim_{R\rightarrow\infty}[I_{D}(h_{1},q_{R})+h_{2}(R)\int_{D}\int_{D}h_{1}(|y_{1}-y_{2}|)dy_{1}dy_{2}]
=I_{D}(h_{1},h_{2}),
\end{multline}
completing the proof.
\end{proof}

\begin{proof} [Proof of Theorem \ref{THM:main}]
Since the logarithmic potential operator is a Hilbert-Schmidt operator, by using bilinear expansion of its iterated kernels (see, for example, \cite{Vla}) we obtain for $p\geq 2$, $p\in\mathbb  N$,
\begin{equation}
\sum_{j=1}^{\infty}\frac{1}{\mu_{j}^{p}(\Omega)}=\int_{\Omega}...
\int_{\Omega}\frac{1}{2\pi}\ln\frac{1}{|y_{1}-y_{2}|}...\frac{1}{2\pi}\ln\frac{1}{|y_{p}-y_{1}|}dy_{1}...dy_{p}.\label{eq15}
\end{equation}
Recalling the inequality \eqref{eqn16} stating that
$$
\int_{\Omega}...
\int_{\Omega}\frac{1}{2\pi}\ln\frac{1}{|y_{1}-y_{2}|}...\frac{1}{2\pi}\ln\frac{1}{|y_{p}-y_{1}|}dy_{1}...dy_{p}\leq$$
\begin{equation}\int_{D}...
\int_{D}\frac{1}{2\pi}\ln\frac{1}{|y_{1}-y_{2}|}...\frac{1}{2\pi}\ln\frac{1}{|y_{p}-y_{1}|}dy_{1}...dy_{p},\label{eq16}
\end{equation}
we obtain
\begin{equation}
\sum_{j=1}^{\infty}\frac{1}{\mu_{j}^{p}(\Omega)}\leq
\sum_{j=1}^{\infty}\frac{1}{\mu_{j}^{p}(D)},\quad p\geq 2,\,\,\, p\in \mathbb N,\label{eq17}
\end{equation}
for any bounded open domain $\Omega\subset \mathbb R^{2}$ with $|\Omega|=|D|$.
Taking even $p$ in \eqref{eq17} we complete the proof of Theorem \ref{THM:main}.
\end{proof}

\begin{proof} [Proof of Theorem \ref{THM:main2}]
The inequality \eqref{eq17} also proves Theorem \ref{THM:main2}
when the logarithmic potential operator is positive (see also Remark \ref{REM:positivity}).
\end{proof}

\begin{rem}
It follows from the properties of the kernel that the Schatten $p$-norm of the operator $\mathcal{L}_\Omega$ is finite when $p>1$ see e.g. the criteria for Schatten classes in terms of the regularity of the kernel in \cite{DR}.
The above techniques do not allow us to prove Theorem \ref{THM:main} for all $p>1$. In view of the Dirichlet Laplacian case, it seems reasonable to conjecture that the Schatten $p$-norm is still maximised on the disc also for all $p>1$.
\end{rem}

\section{On the case of polygons}

We can ask the same question of maximizing the Schatten $p$-norms in the class of polygons with a
given number $n$ of sides. We denote by $\mathcal{P}_{n}$ the class of plane polygons with $n$ edges.
We would like to identify the maximiser for Schatten $p$-norms of the logarithmic potential $\mathcal{L}_\Omega$ in $\mathcal{P}_{n}$.
According to Section \ref{SEC:result},
it is natural to conjecture that it is the $n$-regular polygon. Currently, we can prove this only for $n=3$:

\begin{thm}\label{THM:tri}
The equilateral triangle centred at the origin has the largest Schatten $p$-norm of the operator
$\mathcal{L}_\Omega$ for any even integer $2\leq p< \infty$ among all triangles of a given area.
More precisely, if $\Delta$ is the equilateral triangle centred at the origin, we have
\begin{equation}
\|\mathcal{L}_{\Omega}\|_{p}\leq  \|\mathcal{L}_{\Delta}\|_{p}
\label{3-1}
\end{equation}
for any even integer $2\leq p\leq \infty$ and any bounded open triangle $\Omega$ with $|\Omega|=|\Delta|.$
\end{thm}

Similarly, we have the following analogy of Theorem \ref{THM:main2}:

\begin{thm}\label{THM:tri2}
Let $\Delta$ be an equilateral triangle centred at the origin and let $\Omega$
be a bounded open triangle with $|\Omega|=|\Delta|.$
Assume that the logarithmic potential operator is positive for
$\Omega$ and $\Delta$.
Then
\begin{equation}
\|\mathcal{L}_{\Omega}\|_{p}\leq  \|\mathcal{L}_{\Delta}\|_{p}
\label{3-1}
\end{equation}
for any integer $2\leq p<\infty$.
\end{thm}

Let $u$ be a nonnegative, measurable function on $\mathbb{R}^{2}$, and let $x^{2}$ be a line through the origin of $\mathbb{R}^{2}$. Choose an orthogonal coordinate system in $\mathbb{R}^{2}$ such that the $x^{1}$-axis is perpendicular to $x^{2}$.
\begin{defn}[\cite{BLL}]
A nonnegative, measurable function $u^{\star}(x|x^{2})$ on $\mathbb{R}^{2}$ is called a
{\em Steiner symmetrization} with respect to $x^{2}$ of the function $u(x)$, if $u^{\star}(x^{1},x^{2})$ is a symmetric decreasing rearrangement with respect to $x^{1}$ of $u(x^{1},x^{2})$ for each fixed $x^{2}$.
\end{defn}

The Steiner symmetrization (with respect to the $x^{1}$-axis) $\Omega^{\star}$ of a measurable set $\Omega$ is defined in the following way:
if we write $(x^{1},z)$ with $z\in \mathbb{R}$, and let $\Omega_{z}=\{x^{1}:\; (x^{1},z)\in\Omega\}$, then
$$
\Omega^{\star}=\{(x^{1},z)\in \mathbb{R}\times\mathbb{R}:\; x^{1}\in \Omega^{*}_{z}\}
$$
where $\Omega^{*}$ is a symmetric rearrangement of $\Omega$ (see the proof of Theorem \eqref{THM:1}).
We obtain:

\begin{lem}\label{lem:Riesz}
For a positive function $u$ and a measurable $\Omega\subset \mathbb{R}^{2}$ we have
\begin{multline}
\int_{\Omega}\int_{\Omega}\int_{\Omega}
u(y)\frac{1}{2\pi}\ln\frac{1}{|y-z|}
\frac{1}{2\pi}\ln\frac{1}{|z-x|}u(x)dzdydx\leq
\\
\int_{\Omega^{\star}}
\int_{\Omega^{\star}}\int_{\Omega^{\star}}u^{\star}(y)
\frac{1}{2\pi}\ln\frac{1}{|y-z|}
\frac{1}{2\pi}\ln\frac{1}{|z-x|}u^{\star}(x)dzdydx,\label{nn1}
\end{multline}
where $\Omega^{\star}$ and $u^{\star}$ are Steiner symmetrizations of  $\Omega$ and $u$, respectively.
\end{lem}
\begin{proof} [Proof of Lemma \ref{lem:Riesz}] The proof is based on the proof of Lemma 2 in \cite{CL}.
Let us fix $r_{0}>0$ and consider the function
\begin{equation}
f(r):=\left\{
\begin{array}{ll}
    \frac{1}{2\pi}\ln\,\frac{1}{r},\,\,r\leq r_{0},\\
    \frac{1}{2\pi}\ln\,\frac{1}{r_0}-\frac{1}{2\pi}\int_{r_{0}}^{r}s^{-1}\frac{1+r_{0}^{2}}{1+s^{2}}ds, \,\,r>r_{0}. \\
\end{array}
\right.
\label{EQ:fs1}
\end{equation}
The function $f(r)$ is strictly decreasing and has a limit as $r\rightarrow \infty$ (see the proof of Lemma \ref{lem:n1})
$$
\lim_{r\rightarrow\infty} f(r)=f_{\infty}.
$$
Since $f(r)$ is strictly decreasing $\frac{1}{2\pi}\ln\,\frac{1}{r}-f(r)$ is decreasing.
Thus if we define
$$h_{1}(r)=f(r)-f_{\infty}$$
we have the decomposition
$$\frac{1}{2\pi}\ln\,\frac{1}{r}=h_{1}(r)+h_{2}(r)$$
where $h_{1}$ is positive strictly decreasing function and $h_{2}$
is decreasing. Hence by the Brascamp-Lieb-Luttinger rearrangement inequality
for the Steiner symmetrization (see Lemma 3.2 in \cite{BLL}) we have

\begin{multline}
\int_{\Omega}\int_{\Omega}\int_{\Omega}
u(y)h_{1}(|y-z|)h_{1}(|z-x|)u(x)dzdydx\leq
\\
\int_{\Omega^{\star}}\int_{\Omega^{\star}}
\int_{\Omega^{\star}}u^{\star}(y)h_{1}(|y-z|)h_{1}(|z-x|)
u^{\star}(x)dzdydx.
\end{multline}
Thus it remains to show that
\begin{multline}\label{nnn1}
\int_{\Omega}\int_{\Omega}\int_{\Omega}
u(y)h_{2}(|y-z|)h_{2}(|z-x|)u(x)dzdydx\leq
\\
\int_{\Omega^{\star}}
\int_{\Omega^{\star}}\int_{\Omega^{\star}}
u^{\star}(y)h_{2}(|y-z|)h_{2}(|z-x|)u^{\star}(x)dzdydx
\end{multline}
and
\begin{multline}\label{nnn2}
\int_{\Omega}\int_{\Omega}\int_{\Omega}
u(y)h_{1}(|y-z|)h_{2}(|z-x|)u(x)dzdydx\leq
\\
\int_{\Omega^{\star}}
\int_{\Omega^{\star}}\int_{\Omega^{\star}}
u^{\star}(y)h_{1}(|y-z|)h_{2}(|z-x|)u^{\star}(x)dzdydx,
\end{multline}
which does not follow directly from the Brascamp-Lieb-Luttinger rearrangement inequality
since $h_{2}$ is not positive.
Define for $R>0$
\begin{equation}
q_{R}(r):=\left\{
\begin{array}{ll}
    h_{2}(r)-h_{2}(R),\,\,r\leq R,\\
    0, \,\,r>R, \\
\end{array}
\right.
\label{EQ:fs2}
\end{equation}
and note that by monotone convergence
we have
\begin{equation}
I_{\Omega}(u,h_{2})=\lim_{R\rightarrow\infty}[I_{\Omega}(u,q_{R})
+2h_{2}(R)J_{\Omega}(u,q_{R})+h_{2}^{2}(R)(\int_{\Omega}u(x)dx)^{2}]
\end{equation}
with the notations
\begin{equation}
I_{\Omega}(u,g)=\int_{\Omega}\int_{\Omega}\int_{\Omega}
u(y)g(|y-z|)g(|z-x|)u(x)dzdydx
\end{equation}
and
\begin{equation}
J_{\Omega}(u,g)=\int_{\Omega}\int_{\Omega}\int_{\Omega}
u(y)g(|z-x|)u(x)dzdydx.
\end{equation}
Since $q_{R}$ is positive and nonincreasing and noting that
$$ \int_{\Omega}u(x)dx=\int_{\Omega^{\star}}u^{\star}(x)dx$$
we obtain
$$I_{\Omega}(u,q_{R})\leq I_{\Omega^{\star}}(u^{\star},q_{R}),$$
$$J_{\Omega}(u,q_{R})\leq J_{\Omega^{\star}}(u^{\star},q_{R}),$$
by the Brascamp-Lieb-Luttinger rearrangement inequality.
Therefore,
\begin{multline}
I_{\Omega}(u,h_{2})=\lim_{R\rightarrow\infty}[I_{\Omega}(u,q_{R})
+2h_{2}(R)J_{\Omega}(u,q_{R})+h_{2}^{2}(R)
(\int_{\Omega}u(x)dx)^{2}]\leq\\
\lim_{R\rightarrow\infty}[I_{\Omega^{\star}}(u^{\star},q_{R})
+2h_{2}(R)J_{\Omega^{\star}}(u^{\star},q_{R})+h_{2}^{2}(R)
(\int_{\Omega^{^{\star}}}u^{\star}(x)dx)^{2}]=
I_{\Omega^{\star}}(u^{\star},h_{2}).
\end{multline}
This proves the inequality \eqref{nnn1}.
Similarly, now let us show that the inequality \eqref{nnn2} is valid.
We have
\begin{equation}
\widetilde{I}_{\Omega}(u,h_{2})=\lim_{R\rightarrow\infty}[\widetilde{I}_{\Omega}(u,q_{R})
+h_{2}(R)\widetilde{J}_{\Omega}(u,h_{1})]
\end{equation}
with the notations
\begin{equation}
\widetilde{I}_{\Omega}(u,g)=\int_{\Omega}\int_{\Omega}\int_{\Omega}
u(y)h_{1}(|y-z|)g(|z-x|)u(x)dzdydx
\end{equation}
and
\begin{equation}
\widetilde{J}_{\Omega}(u,h_{1})=\int_{\Omega}\int_{\Omega}\int_{\Omega}
u(y)h_{1}(|y-x|)u(x)dzdydx.
\end{equation}
Since both $q_{R}$ and $h_{1}$ are positive and nonincreasing
\begin{equation}\widetilde{I}_{\Omega}(u,q_{R})\leq \widetilde{I}_{\Omega^{\star}}(u^{\star},q_{R}),\end{equation}
\begin{equation}\widetilde{J}_{\Omega}(u,h_{1})\leq \widetilde{J}_{\Omega^{\star}}(u^{\star},h_{1}),\end{equation}
by the Brascamp-Lieb-Luttinger rearrangement inequality.
Therefore, we obtain
$$\widetilde{I}_{\Omega}(u,h_{2})=\lim_{R\rightarrow\infty}[\widetilde{I}_{\Omega}(u,q_{R})
+h_{2}(R)\widetilde{J}_{\Omega}(u,h_{1})]\leq$$
$$\lim_{R\rightarrow\infty}[\widetilde{I}_{\Omega^{\star}}(u^{\star},q_{R})
+h_{2}(R)\widetilde{J}_{\Omega^{\star}}(u^{\star},h_{1})]=
\widetilde{I}_{\Omega^{\star}}(u^{\star},h_{2}).$$
This proves the inequality \eqref{nnn2}.
\end{proof}

Lemma \ref{lem:Riesz} implies the following analogy of the P{\'o}lya theorem \cite{Po} for the operator $\mathcal{L}_\Omega$.

\begin{thm} \label{THM:2}
The equilateral triangle  $\Delta$ centred at the origin is a minimiser of the first characteristic number of the logarithmic potential $\mathcal{L}_\Omega$ among all triangles of a given area, i.e.
$$\frac{1}{|\mu_{1}(\Omega)|}\leq \frac{1}{|\mu_{1}(\Delta)|}$$
for any triangle $\Omega\subset \mathbb R^{2}$ with $|\Omega|=|\Delta|.$
\end{thm}

\begin{rem}\label{REM:op2}
In other words Theorem \ref{THM:2} says that the operator norm of
$\mathcal{L}_\Omega$ is maximised in an equilateral triangle among all triangles of a given area.
\end{rem}

\begin{proof} [Proof of Theorem \ref{THM:2}]
By Theorem \ref{THM:1} and Lemma \ref{lem:1} the first characteristic number $\mu_{1}$ of the operator $\mathcal{L}_\Omega$ is positive and simple; the corresponding
eigenfunction $u_{1}$ can be chosen positive in $\Omega$.
Using the fact that by applying
a sequence of the Steiner symmetrizations with respect to the mediator of each side, a given triangle converges to an equilateral one
(see e.g. Figure 3.2. in \cite{He}), from \eqref{nn1} we obtain

\begin{multline}
\int_{\Omega}\int_{\Omega}\int_{\Omega}
u_{1}(y)\frac{1}{2\pi}\ln\frac{1}{|y-z|}
\frac{1}{2\pi}\ln\frac{1}{|z-x|}u_{1}(x)dzdydx\leq\\
\int_{\Delta}
\int_{\Delta}\int_{\Delta}
u_{1}^{\star}(y)\frac{1}{2\pi}\ln\frac{1}{|y-z|}
\frac{1}{2\pi}\ln\frac{1}{|z-x|}u_{1}^{\star}(x)dzdydx.\label{nn4}
\end{multline}

Therefore, from \eqref{nn4} and the variational principle for the positive operator $\mathcal{L}^{2}_\Delta$, we get
$$
\mu_{1}^{2}(\Omega) =\frac{\int_{\Omega}|u_{1}(x)|^{2}dx}{\int_{\Omega}\int_{\Omega}\int_{\Omega}
u_{1}(y)\frac{1}{2\pi}\ln\frac{1}{|y-z|}
\frac{1}{2\pi}\ln\frac{1}{|z-x|}u_{1}(x)dzdydx}\geq$$
$$\frac{\int_{\Delta}|u^{\star}_{1}(x)|^{2}dx}
{\int_{\Delta}\int_{\Delta}\int_{\Delta}
u^{\star}_{1}(y)\frac{1}{2\pi}\ln\frac{1}{|y-z|}
\frac{1}{2\pi}\ln\frac{1}{|z-x|}u^{\star}_{1}(x)dzdydx}\geq
$$
$$
\inf_{v\in
L^{2}(\Delta)}\frac{\int_{\Delta}|v(x)|^{2}dx}{\int_{\Delta}\int_{\Delta}\int_{\Delta}
v(y)\frac{1}{2\pi}\ln\frac{1}{|y-z|}
\frac{1}{2\pi}\ln\frac{1}{|z-x|}v(x)dzdydx}=\mu_{1}^{2}(\Delta).
$$
Here we have used the fact that the Steiner symmetrization preserves the $L^{2}$-norm.
\end{proof}

\begin{proof} [Proofs of Theorem \ref{THM:tri} and Theorem \ref{THM:tri2}]
The proofs of Theorem \ref{THM:tri} and Theorem \ref{THM:tri2} rely on the same techniques as the proofs of Theorem \ref{THM:main} and Theorem \ref{THM:main2} with
the difference that now the Steiner symmetrization is used.
Since the Steiner symmetrization has the same property \eqref{eq16} (cf. Lemma 3.2 in \cite{BLL})
as the symmetric-decreasing rearrangement, it is
clear that any Steiner symmetrization increases (or at least does not decrease) the Schatten $p$-norms
 for even integers $p\geq 2$.
Thus, for the proof we only need to recall the fact that a sequence of Steiner symmetrizations with respect to
the mediator of each side, a given triangle converges to an equilateral one.
The rest of the proof is exactly the same as the proofs of Theorem \ref{THM:main} and Theorem \ref{THM:main2}.
\end{proof}

\begin{rem} A sequence of three Steiner symmetrizations allows us to transform any quadrilateral
into a rectangle (see Figure 3.3. in \cite{He}). Therefore, it suffices to look at the maximization problem among rectangles for $\mathcal{P}_{4}$. Unfortunately, for $\mathcal{P}_{5}$ (pentagons and others), the Steiner symmetrization
increases, in general, the number of sides. This prevents us from using the same technique for general polygons.
\end{rem}

%\bibliographystyle{alphaabbr}
%\bibliography{Durvudkhan-Niyaz-15-03-22}
%\end{document}

\end{document}